\documentclass[a4paper]{amsart}
\usepackage{fullpage}
\usepackage{amsmath,amssymb,mathtools,booktabs,tikz}
\usetikzlibrary{arrows.meta,calc,decorations.markings,positioning}
\usepackage[T1]{fontenc}
\usepackage{lmodern}
\usepackage{microtype}
\usepackage{needspace}
\usepackage{tabularx,array}
\usepackage[hypertexnames=false]{hyperref}
\usepackage{doi}
\usepackage{float}
\hypersetup{
    colorlinks=true,
    linkcolor=blue,
    citecolor=blue,
    urlcolor=blue
}
\hypersetup{
    pdftitle={Entire solutions of higher order Briot--Bouquet equations},
    pdfauthor={Alexandre Eremenko, Tuen Wai Ng, Teng Zhang}
}
\numberwithin{equation}{section}
\newtheorem{thm}{Theorem}[section]
\newtheorem{lem}[thm]{Lemma}
\newtheorem{prop}[thm]{Proposition}
\newtheorem{cor}[thm]{Corollary}
\theoremstyle{remark}

\theoremstyle{plain}

\newtheorem*{picard}{Picard's Theorem}

\newcommand{\C}{\mathbb C}
\newcommand{\Z}{\mathbb Z}
\newcommand{\N}{\mathbb N}

\newcommand{\Chat}{\widehat{\mathbb C}}
\newcommand{\D}{\mathbb D}
\newcommand{\W}{\mathcal W}
\newcommand{\B}{\mathcal B}
\DeclareMathOperator{\conv}{conv}

\DeclareMathOperator{\Rea}{Re}
\newcommand{\ee}{\mathrm e}
\newcommand{\dd}{\,\mathrm d}
\newcommand{\ii}{\mathrm i}

\begin{document}

\title[Entire solutions of higher order Briot--Bouquet equations]
{Entire solutions of higher order
Briot--Bouquet equations}
\author[A.~Eremenko]{Alexandre Eremenko}
\address{Mathematics Department, Purdue University,
West Lafayette, IN 47907, USA}
\email{eremenko@purdue.edu}

\author[T.~W.~Ng]{Tuen Wai Ng}
\address{Department of Mathematics, The University of Hong Kong,
Pokfulam, Hong Kong}
\email{ntw@maths.hku.hk}

\author[T.~Zhang]{Teng Zhang}
\address{School of Mathematics and Statistics,
Xi'an Jiaotong University, Xi'an 710049, P.~R.~China}
\email{teng.zhang@stu.xjtu.edu.cn}


\begin{abstract}
For a meromorphic function $f$ in the complex plane, we determine when
$f$ and $f^{(k)}$ are algebraically dependent. This happens precisely
when $f$ belongs to the class $\W$ of elliptic functions and their
degenerates, or satisfies a linear differential equation
$f^{(k)}+af+b=0$ with constant coefficients. The case of functions
having a pole was settled previously. We complete the proof by
treating entire solutions. 
\end{abstract}

\maketitle
\section{Introduction}\label{sec:introduction}

The classical theorem of Briot and Bouquet \cite{BB56},
dating back to the mid-19th century,  asserts that all meromorphic
solutions of differential equations of the form
\[
 F(y,y')=0,\qquad F\in\C[X,Y],
\]
are functions of the class $W$, which consists of all elliptic functions, rational
functions and functions of the form $R(\exp(cz))$, where $R$ is rational and $c$ is
a constant. An equivalent formulation which was preferred in the 19th century
is: \emph{An inverse of an Abelian integral is meromorphic in the plane only if it
is an elliptic function, possibly degenerate.}

The class $W$ can also be characterized as the class of all meromorphic functions
which admit an algebraic addition theorem. This is a theorem of Weierstrass
(see, for example, \cite[Chapter.~XI, \S~62]{Akh90}), and the class is named in his honor.

Picard \cite{Pic80}, after proving his famous theorems on omitted values and
on the uniformization of algebraic curves by meromorphic functions in the
plane (which we state precisely below) gives the following theorem as his first
application:

\begin{picard}
All meromorphic solutions of differential equations
\[
 F(y,y'')=0,\qquad F\in\C[X,Y]
\]
belong to the class $W$.
\end{picard}

Almost a century later, Einar Hille began to study analogous differential
equations of higher order \cite{Hil76,Hil78a,Hil78b,Hil78c,Hil78d}
\begin{equation}\label{eq:bb}
 F(y,y^{(k)})=0,\qquad F\in\C[X,Y], \quad k\in \mathbb N_+
\end{equation}
which he called the higher order Briot--Bouquet equations. He did not know
Picard's Theorem, and stated it as a conjecture, proving it under some additional
conditions. Hille's conjecture was proved by Bank and Kaufman
\cite{BK81}, about 100 years after Picard's work.
Hille also studied necessary conditions for the existence of
meromorphic solutions and their growth. For example,
\cite{Hil78d} gives an upper bound of one
for the order of entire solutions under certain degree assumptions.
He also studied the singularities of non-meromorphic solutions.

Hille's ideas were developed in \cite{Ere82}, where a conjecture that all meromorphic
solutions of equations \eqref{eq:bb} belong to $W$ was stated. This conjecture has
an evident exception, namely
\[
 y'''-y=0
\]
which has solutions outside $\W$. So the correct conjecture is that all
meromorphic solutions of \eqref{eq:bb} either belong to $W$ or satisfy linear
differential equations of the form $y^{(k)}+ay+b=0$. This was proposed in
\cite{ELN09}, where the conjecture was proved for meromorphic functions with at least
one pole. In this paper, we complete the proof of this conjecture.

In \cite{Ere82}, necessary conditions on $F$ for existence of meromorphic solutions
were proved, as well as the fact that when the curve $F(X,Y)=0$ has genus
1, then all meromorphic solutions are elliptic functions.

It was also proved in \cite{Ere82} that for odd $k$, all meromorphic solutions having
one pole must belong to the class $W$. This was later extended to all $k$ in \cite{ELN09}.
So the only remaining question is about entire solutions.

In this paper we complete this line by treating the case of entire solutions.
In addition we give a new proof of the necessary conditions on the polynomial
$F$ in \eqref{eq:bb} from \cite[Theorem~1]{Ere82}. The reasons for including this new proof are the
following.

The original proof of the necessary conditions in \cite{Ere82} contained a gap; a
generalization of the Wiman--Valiron theory which was unproved at that time
was used. This gap was removed by bypassing this unproved statement in
the first-named author's thesis \cite{Ere86}, but this thesis was never published. 22
years later, Bergweiler, Rippon and Stallard \cite{BRS08} proved the statement whose
proof was missing in \cite{Ere82}, thus justifying all arguments in \cite{Ere82}. Since the result
in \cite{BRS08} is highly non-trivial, we include here the proof from the thesis \cite{Ere86} which
does not rely on \cite{BRS08}.

The second reason is to make explicit the reduction used below.
Proposition~\ref{prop:necessary} and its corollaries reduce the
remaining entire case to \eqref{eq:yanchenko}; this reduction was
recorded in \cite{Ere24}.
Yanchenko \cite{Yan22} proved the conjecture for entire
solutions when the highest homogeneous part of $F$ is a product
of distinct linear factors.

The main result of this paper is

\begin{thm}\label{thm:pdf-main}
Every transcendental entire solution of an irreducible equation
\eqref{eq:bb} is of the form $R(e^{cz})$, where $R$ is rational, and $c\in\C^*$, except for the
case when \eqref{eq:bb} is a linear equation.
\end{thm}

\medskip
\noindent\textbf{Organization of the paper.}
In Section~\ref{sec:reduction} we establish necessary conditions on the
polynomial in the differential equation and reduce the remaining entire
case to polynomial compositions. Section~\ref{sec:borel} develops the
Laplace--Borel transform tools and separates the characteristic
exponentials from a remainder with a smaller conjugate indicator diagram.
In Section~\ref{sec:rigidity} we prove the polynomial-composition rigidity
theorem using Rouch\'e's theorem and a translation argument.
Section~\ref{sec:classification} completes the classification of
meromorphic solutions and describes the exceptional linear family and
its intersection with $\W$.

\medskip
\noindent\textbf{Acknowledgments and AI tools disclosure.}
We used ChatGPT 6.1 for English-language editing and proofreading, as well as to explore possible approaches to parts of the proofs. The principal ideas underlying the results were suggested by ChatGPT 6.1.
The authors developed, carefully verified, and rewrote all the mathematical arguments, simplifying most of the proofs in the process.

We also learned that D.~Marques \cite{Mar26} had independently
obtained the same result.

\section{\texorpdfstring
  {Necessary conditions on the polynomial $F$ for the existence of a
   meromorphic solution of \eqref{eq:bb} and reduction of the problem}
  {Necessary conditions and reduction}}
\label{sec:reduction}

To state the result we need some preliminaries. We will always assume
that the polynomial $F$ in \eqref{eq:bb} is irreducible, which does not restrict generality.

Let $\Gamma$ be the normalization of the projective closure of
$F(X,Y)=0$, and let $p,q:\Gamma\to\Chat$ be the meromorphic
coordinate functions, so that $F(p,q)=0$.

At a pole $s$ of $p$, choose a local parameter $\zeta$ such that
$p=\zeta^{-r_s}$, where $r_s\geq1$. If $q\not\equiv0$, write
\[
 q(\zeta)=a_s\zeta^{-t_s}(1+O(\zeta)),
 \qquad a_s\neq0,\quad t_s\in\Z.
\]
The corresponding Puiseux expansion has leading term
$a_s p^{\kappa(s)}$, where
\[
 \kappa(s)=\frac{t_s}{r_s}.
\]
We call $\kappa(s)$ the exponent at $s$. Different determinations
of a Puiseux expansion belonging to the same local branch represent
the same point of $\Gamma$.

A meromorphic solution $y$ of \eqref{eq:bb} induces a holomorphic map $\phi:\C\to\Gamma$
such that $y=p\circ\phi$ and $y^{(k)}=q\circ\phi$.

Picard's theorem on the uniformization says that the existence of a non-constant
holomorphic map $\phi:\C\to\Gamma$ implies that the genus $g$ of $\Gamma$ is 0 or 1.
Picard's theorem on the omitted values implies that the non-constant map $\phi$
can omit at most two points in $\Gamma$, and omitted points can occur only when
$g=0$. First we prove

\begin{prop}\label{prop:entire-type}
Every entire solution of \eqref{eq:bb} is of exponential type.
\end{prop}

\begin{proof}
Polynomial solutions are of exponential type, so suppose that $y$
is transcendental. Write
\begin{equation}\label{eq:pdf-polynomial}
 F_m(y)(y^{(k)})^m+\cdots+F_0(y)=0,
\end{equation}
where the $F_j$ are polynomials.

Let $M(r)=\max_{|z|=r}|y(z)|$, and let $\mu(r)$ and $\nu(r)$
be the maximum term and central index of $y$. Classical
Wiman--Valiron theory \cite[Chapter.~I]{Wit55} gives an exceptional
set $E$ of finite logarithmic measure and points $z_r$, with
$|z_r|=r$ and $|y(z_r)|=M(r)$, such that
\begin{equation}\label{eq:wv}
 y^{(j)}(z_r)=(1+o(1))
 \left(\frac{\nu(r)}{z_r}\right)^j y(z_r),
 \qquad 1\leq j\leq k,
\end{equation}
as $r\to\infty$ outside $E$. Enlarging $E$ if necessary, we also have
\begin{equation}\label{eq:pdf-central-index}
 \nu(r)\leq(\log M(r))^2,\qquad r\notin E.
\end{equation}
Indeed, $d\log\mu(r)/d\log r=\nu(r)$ almost everywhere, so the set
where $\nu(r)>(\log\mu(r))^2$ has finite logarithmic measure.

Since $y$ is transcendental, $M(r)/r^A\to\infty$ for every $A>0$.
Together with \eqref{eq:pdf-central-index}, this implies that, for
every $\varepsilon>0$ and all sufficiently large $r\notin E$,
\[
 M(r)^{-\varepsilon}
 \leq\left(\frac{\nu(r)}r\right)^k
 \leq M(r)^\varepsilon.
\]

The algebraic relation \eqref{eq:pdf-polynomial} has finitely many
Puiseux branches above $X=\infty$. Their leading terms are
$a_sX^{\kappa(s)}$, with $a_s\neq0$. At $X=y(z_r)$, comparison with
\eqref{eq:wv} gives, for the branch containing
$(y(z_r),y^{(k)}(z_r))$,
\[
 \left(\frac{\nu(r)}r\right)^k
   =(|a_s|+o(1))M(r)^{\kappa(s)-1}.
\]
The preceding bounds and the finiteness of the set of branches show
that only branches with $\kappa(s)=1$ can occur for sufficiently
large $r\notin E$. Hence $\nu(r)=O(r)$ outside $E$.

For every sufficiently large $r$, the interval $[r,2r]$ contains
a point outside $E$. Monotonicity of $\nu$ therefore gives
$\nu(r)=O(r)$ for all large $r$. Consequently,
\[
 \log\mu(r)=\log\mu(r_0)+\int_{r_0}^r\frac{\nu(t)}t\,dt=O(r).
\]
Finally, $M(r)\leq2\mu(2r)$, by comparison of the Taylor series
with a geometric series. Thus $\log M(r)=O(r)$.
\end{proof}
We now return to the algebraic curve $\Gamma$ and describe how
the poles of its coordinate functions are constrained by the
existence of a transcendental meromorphic solution.
The following proposition will be used to reduce the entire
case to polynomial compositions. Its proof is given after
Lemma~\ref{lem:finite-fiber}.
\begin{prop}\label{prop:necessary}
Suppose that \eqref{eq:bb} has a transcendental meromorphic
solution $y$ and that $\Gamma$ has genus zero. Then $p$ and $q$
have the same set $V\subset\Gamma$ of poles.
For every $s\in V$, either $\kappa(s)=1$ or
$\kappa(s)=1+k/n$ for some $n\in\N$.

If $y$ has a pole of order $n$, then $\kappa(s)=1+k/n$ at its
image $s\in V$. Conversely, if $\kappa(s)=1+k/n$ at some
$s\in V$, then $y$ has infinitely many poles of order $n$.

If $s\in V$ is omitted by $\phi$, then $\kappa(s)=1$.
\end{prop}
Combining these conditions with Picard's theorems gives the
following corollary for entire solutions.
\begin{cor}\label{cor:entire-poles}
If \eqref{eq:bb} has a transcendental entire solution, then
$\Gamma$ has genus zero, $|V|\in\{1,2\}$, and
$\kappa(s)=1$ for every $s\in V$.
\end{cor}

\begin{proof}
The nonconstant meromorphic function $p$ has a pole, and every
pole of $p$ is omitted by $\phi$. Picard's theorems therefore give
genus zero and at most two such poles. Now apply
Proposition~\ref{prop:necessary}.
\end{proof}

\begin{cor}\label{cor:two-poles}
If \eqref{eq:bb} has a transcendental entire solution $y$ and $|V|=2$,
then $y(z)=R(e^{cz})$, where $R$ is rational and $c\in\C^*$.
\end{cor}

\begin{proof}
Choose a coordinate on $\Gamma$ that sends the two points of $V$
to $0$ and $\infty$. Then $y=R(e^g)$, where $g$ is entire and
$R$ is rational with poles at both $0$ and $\infty$ and no others.
For some constants $A,B>0$,
\[
 \max\{|w|,|w|^{-1}\}\leq A(1+|R(w)|)^B,
 \qquad w\in\C^*.
\]
Proposition~\ref{prop:entire-type} now gives
$|\operatorname{Re}g(z)|\leq A_1+B_1|z|$.
The Borel--Carath\'eodory theorem and Cauchy's estimates imply
that $g$ is affine. Its slope is nonzero because $y$ is nonconstant.
Absorbing its constant term into $R$ proves the assertion.
\end{proof}
The following corollary follows from Eremenko's necessary
conditions \cite[Theorem~1]{Ere82}. The reduction was stated
explicitly in \cite{Ere24}.
\begin{cor}\label{cor:yanchenko}
Suppose that \eqref{eq:bb} has a transcendental entire solution
and $|V|=1$. Then either \eqref{eq:bb} is linear or, after
multiplication of $F$ by a nonzero constant, it has the form
\begin{equation}\label{eq:yanchenko}
 (y^{(k)}-cy)^n=P_n(y,y^{(k)}),
 \qquad \deg P_n\leq n-1,
\end{equation}
where $n\geq2$ and $c\in\C^*$.
\end{cor}

\begin{proof}
Let $s$ be the unique point of $V$. By
Proposition~\ref{prop:necessary}, $p$ and $q$ have equal pole orders
at $s$, and $q/p\to c\neq0$ there. The projective closure of
$F(X,Y)=0$ therefore has exactly one point at infinity, namely
$[1:c:0]$. If $F_d$ is the highest homogeneous part of $F$,
its only zero on the projective line is $[1:c]$. Hence
\[
 F_d(X,Y)=a(Y-cX)^d,\qquad a\neq0.
\]
If $d=1$, the equation is linear. Otherwise, divide by $a$ and
move the terms of lower degree to the other side.
\end{proof}

Thus the only case which remains to consider is entire solutions of equation
\eqref{eq:yanchenko}.

\begin{lem}\label{lem:finite-fiber}
Suppose that $y$ is transcendental meromorphic and $\Gamma$ has
genus zero. If $s\in\Gamma$ is a pole of at least one of $p,q$
and $\phi^{-1}(s)$ is finite, then $p$ and $q$ both have poles
at $s$, of the same order.
\end{lem}

\begin{proof}
Choose a coordinate $t:\Gamma\to\Chat$ whose only pole is at $s$,
and put $H_0=t\circ\phi$. Then $H_0$ is transcendental meromorphic
with finitely many poles, and
\[
 y=R(H_0),\qquad y^{(k)}=T(H_0)
\]
for nonconstant rational functions $R,T$.
At infinity, write
\begin{equation}\label{eq:pdf-local-composition}
 R(w)-b_0=A w^a(1+O(w^{-1})),\qquad
 T(w)=B w^b(1+O(w^{-1})),
\end{equation}
where $AB\neq0$, $a\in\Z\setminus\{0\}$ and $b\in\Z$.
Here $b_0=0$ if $R$ has a pole at infinity, and otherwise
$b_0=R(\infty)$. At least one of $a,b$ is positive.

Subtract the rational principal parts of $H_0$ to obtain a
transcendental entire function $H$. Apply the Wiman--Valiron
estimates used in Proposition~\ref{prop:entire-type} to $H$.
Writing $M(r)=M(r,H)$ and $\nu(r)$ for its central index, we obtain
points $z_r$, $|z_r|=r$, outside an exceptional set of finite
logarithmic measure, such that
\[
 H_0(z_r)=(1+o(1))H(z_r),\qquad
 H_0^{(j)}(z_r)=(1+o(1))
 \left(\frac{\nu(r)}{z_r}\right)^j H_0(z_r),
 \quad 1\leq j\leq k.
\]
The rational terms that were subtracted are negligible here,
since $M(r)$ grows faster than every power of $r$.

Fa\`a di Bruno's formula, together with
\eqref{eq:pdf-local-composition}, gives
\[
 y^{(k)}(z_r)
  =(Aa^k+o(1))H_0(z_r)^a
    \left(\frac{\nu(r)}{z_r}\right)^k.
\]
The coefficient $a^k$ is the coefficient obtained by applying
$(w\,d/dw)^k$ to $w^a$; in particular, it is nonzero.
Comparison with $T(H_0(z_r))$ yields
\[
 \left(\frac{\nu(r)}r\right)^k
  =\left(\left|\frac{B}{Aa^k}\right|+o(1)\right)M(r)^{b-a}.
\]
As in the proof of Proposition~\ref{prop:entire-type}, the left-hand
side lies between $M(r)^{-\varepsilon}$ and $M(r)^\varepsilon$
for every $\varepsilon>0$ and all sufficiently large admissible $r$.
Thus $a=b$. Since at least one is positive, both are positive,
and the two pole orders are equal.
\end{proof}

\begin{proof}[Proof of Proposition~\ref{prop:necessary}]
Suppose first that $s$ is a pole of at least one of $p,q$ and
$s=\phi(z_0)$. Since $y$ and $y^{(k)}$ have the same poles,
$s$ is a pole of both coordinate functions. If $y$ has a pole
of order $n$ at $z_0$, then
\[
 y(z)\sim A(z-z_0)^{-n},\qquad
 y^{(k)}(z)\sim
 A(-1)^k n(n+1)\cdots(n+k-1)(z-z_0)^{-n-k}.
\]
Let $e$ be the local degree of $\phi$ at $z_0$, and let $r_s,t_s$
be the pole orders of $p,q$ at $s$. Then
$n=e r_s$ and $n+k=e t_s$, so $\kappa(s)=1+k/n$.

If a pole of one coordinate function were not a pole of the other,
it could not be attained by $\phi$. Lemma~\ref{lem:finite-fiber}
would then give a contradiction. Thus the pole sets coincide.

For a common pole with finite fibre, the lemma gives
$\kappa(s)=1$; the calculation above shows that this fibre must
in fact be empty. Every other common pole has infinitely many
preimages, and at each of them the pole order of $y$ equals
$k/(\kappa(s)-1)$. This proves all the assertions.
\end{proof}

It remains to consider a transcendental entire solution in the
case $|V|=1$. Choose a coordinate $t$ on $\Gamma$ whose only pole
is the point of $V$, and put $h=t\circ\phi$. Then $h$ is entire
and
\[
 y=P(h),\qquad y^{(k)}=Q(h),
\]
where $P,Q$ are polynomials of the same degree $N$.
If $N=1$, the equation is linear. Otherwise $N\geq2$.
Let $c\neq0$ be the leading coefficient of $Q$ divided by that
of $P$. Then $S_0=Q-cP$ has degree at most $N-1$.

Choose $\alpha\neq0$ with $\alpha^k c=1$, and set
$f(z)=y(\alpha z)$ and $\widetilde h(z)=h(\alpha z)$. Then
\[
 f=P(\widetilde h),\qquad
 f^{(k)}-f=\alpha^k S_0(\widetilde h).
\]
Thus Theorem~\ref{thm:pdf-main} follows from the next proposition.

\begin{prop}\label{prop:pdf-composition}
Let $h$ be entire, let $P,S\in\C[w]$, and suppose that
$n=\deg P\geq2$ and $\deg S\leq n-1$.
If $f=P(h)$ is of exponential type and
\begin{equation}\label{eq:composition}
 f^{(k)}-f=S(h),
\end{equation}
then $f\in\W$.
\end{prop}

We prove Proposition~\ref{prop:pdf-composition} by applying the
Laplace--Borel transform to \eqref{eq:composition-recalled}.

\section{Application of Laplace--Borel transform}\label{sec:borel}

We use $D=\dd/\dd z$. The finite exponential type required below is
provided by Proposition~\ref{prop:entire-type}. We recall the
Laplace--Borel transform and the classical indicator-diagram
description, in the form needed for the rigidity argument.

For a nonzero entire function $u$ of exponential type, define
\[
 I_u(\theta)=\limsup_{r\to\infty}
             \frac{\log|u(r\ee^{\ii\theta})|}{r}.
\]
Its Laplace transform along the ray of angle $\theta$ is
\begin{equation}\label{eq:laplace-ray}
 \B u(\xi)=\int_0^{\infty\ee^{\ii\theta}}
                 \ee^{-\xi z}u(z)\dd z,
 \qquad \Rea(\xi\ee^{\ii\theta})>I_u(\theta).
\end{equation}
It is holomorphic in this half-plane. For sufficiently large
$\Rea(\xi\ee^{\ii\theta})$, termwise integration gives the germ
at infinity
\[
 \B u(\xi)=\sum_{j=0}^{\infty}
                      \frac{u^{(j)}(0)}{\xi^{j+1}}.
\]
The ray integrals in \eqref{eq:laplace-ray} agree by analytic
continuation. Rotation of the ray continues this germ to the
complement of a compact convex set $K_u$, the \emph{conjugate
indicator diagram}. Equivalently, $K_u$ is the smallest compact
convex set outside which the Borel germ has a single-valued
holomorphic continuation; it is nonempty for $u\neq0$.

With the support-function convention
\[
 H_K(z)=\max_{\xi\in K}\Rea(\xi z),
\]
P\'olya's theorem gives
\begin{equation}\label{eq:indicator}
 I_u(\theta)=H_{K_u}(\ee^{\ii\theta}).
\end{equation}
This is the ordinary upper radial indicator; no regularity of
directional growth is assumed. For this formulation and the
inversion formula below, see
\cite[formula~(3)]{Bei14}.
For every $\varepsilon>0$ we also have the uniform estimate
\begin{equation}\label{eq:uniform}
 |u(z)|\leq C_\varepsilon
       \exp\bigl(H_{K_u}(z)+\varepsilon|z|\bigr)
 \qquad(z\in\C).
\end{equation}
The function is recovered by
\begin{equation}\label{eq:polya}
 u(z)=\frac{1}{2\pi\ii}\int_\gamma
                  \B u(\xi)\ee^{\xi z}\dd\xi,
\end{equation}
where $\gamma$ is a positively oriented contour surrounding $K_u$
and lying in its complement. Formula~\eqref{eq:uniform} follows
by taking $\gamma$ in a sufficiently small convex neighborhood
of $K_u$. We set $K_0=\varnothing$ and treat the zero function
separately when using an indicator.

A useful immediate consequence is the following carrier criterion.
If $u\neq0$ and the Borel germ extends outside a compact convex set
$A$, then $K_u\subset A$. Conversely, this inclusion supplies that
continuation. Formula~\eqref{eq:polya}, with contours in arbitrarily
small convex neighborhoods of $A$, also gives
\eqref{eq:uniform} with $H_A$ in place of $H_{K_u}$. In particular,
derivatives and fixed translations preserve the space of functions
whose diagrams lie in $A$, and so do finite linear combinations.
For derivatives this also follows from Cauchy's estimate with radius
one; for translations use
$H_A(z+T)\leq H_A(z)+H_A(T)$.

Set
\[
 L=D^k-1,\qquad \mu_k=\{\lambda\in\C:\lambda^k=1\}.
\]
The next lemma isolates the operator property that will turn a
translation estimate into an exact period.

\begin{lem}\label{lem:operator}
If $u$ is an entire function of exponential type, then
\begin{equation}\label{eq:borel-identity}
 \B(Lu)(\xi)=(\xi^k-1)\B u(\xi)
       -\sum_{j=0}^{k-1}u^{(j)}(0)\xi^{k-1-j}.
\end{equation}
Consequently,
\begin{equation}\label{eq:carrier-inclusion}
 K_u\subset\conv(K_{Lu}\cup\mu_k).
\end{equation}
If $u\neq0$ and $K_u\subset\D$, where $\D=\{\xi:|\xi|<1\}$,
then $Lu\neq0$ and
\begin{equation}\label{eq:diagram-preserved}
 K_{Lu}=K_u.
\end{equation}
\end{lem}

\begin{proof}
Reindexing the series defining the Borel transform proves
\eqref{eq:borel-identity} near infinity. Dividing by $\xi^k-1$
extends $\B u$ outside $K_{Lu}$ except possibly at the points of
$\mu_k$, which proves \eqref{eq:carrier-inclusion} by the carrier
criterion.

Differentiation and subtraction do not enlarge a convex carrier, so
$K_{Lu}\subset K_u$. Assume now that $K_u\subset\D$ and $Lu\neq0$.
The same quotient continues $\B u$ to $\C\setminus K_{Lu}$ apart
from possible poles at $\mu_k$. These points lie outside $K_u$,
where the original continuation of $\B u$ is already holomorphic.
The two continuations agree near infinity and hence throughout their
common exterior domain. The apparent poles at $\mu_k$ are therefore
removable. The carrier criterion gives $K_u\subset K_{Lu}$, proving
\eqref{eq:diagram-preserved}.

If instead $Lu=0$, formula~\eqref{eq:borel-identity} makes $\B u$ a
rational function whose only possible poles lie in $\mu_k$. The same
removability argument eliminates all of them. Since $\B u(\xi)\to0$
at infinity, $\B u=0$, and the defining series gives $u=0$. This
contradicts the hypothesis.
\end{proof}

The following sublinear estimate separates the characteristic
exponentials from a remainder in a strictly contracted diagram:
\begin{equation}\label{eq:sublinear}
 |f^{(k)}(z)-f(z)|\leq C\bigl(1+|f(z)|^\rho\bigr),
 \qquad z\in\C,\qquad 0\leq\rho<1.
\end{equation}

\begin{prop}\label{prop:decomposition}
Let $f$ be a nonconstant entire function of finite exponential type
satisfying \eqref{eq:sublinear}, and put
\[
 g=Lf,\qquad K=\conv(K_f\cup\{0\}).
\]
There are a nonempty set $E\subset\mu_k$, nonzero constants
$c_\lambda$ for $\lambda\in E$, and an entire function $v$ of
exponential type such that
\begin{equation}\label{eq:decomposition}
 f(z)=\sum_{\lambda\in E}c_\lambda\ee^{\lambda z}+v(z),
 \qquad K_v\subset\rho K,
\end{equation}
and
\begin{equation}\label{eq:hull-control}
 K\subset\conv(E\cup\{0\})\subset\overline\D.
\end{equation}
\end{prop}

\begin{proof}
First, the exponential type of $f$ is positive. At a point $z_r$
where $|f(z_r)|=M(r,f)$, inequality~\eqref{eq:sublinear} gives
$|f^{(k)}(z_r)|\geq M(r,f)/2$ for all sufficiently large $r$.
Cauchy's estimate on a circle of fixed radius $d_0$ centered at
$z_r$ then gives
\[
 M(r+d_0,f)\geq\frac{d_0^k}{2k!}M(r,f).
\]
Choose $d_0$ so that $d_0^k/(2k!)>1$ and iterate. This gives a
positive exponential lower bound.

If $g\neq0$, inequality~\eqref{eq:sublinear} and the exact indicator
formula~\eqref{eq:indicator} give
\[
 H_{K_g}(\ee^{\ii\theta})
 \leq\rho\max\{H_{K_f}(\ee^{\ii\theta}),0\}
 =\rho H_K(\ee^{\ii\theta}).
\]
Thus $K_g\subset\rho K$; this inclusion also holds when $g=0$.
Together with \eqref{eq:carrier-inclusion}, it follows that
\[
 K\subset\conv(\rho K\cup\mu_k\cup\{0\}).
\]
Let $M=\conv(\mu_k\cup\{0\})$. In any direction, the preceding
inclusion says $H_K\leq\max\{\rho H_K,H_M\}$. Since $H_K\geq0$
and $\rho<1$, this implies $H_K\leq H_M$. Hence $K\subset M$
and, in particular, $K\subset\overline\D$.

The function $\B g$ is holomorphic outside $\rho K$. As the roots of
$\xi^k-1$ are simple and lie outside $\rho K$, the quotient in
\eqref{eq:borel-identity} shows that $\B f$ is meromorphic there,
with at most simple poles at the points of $\mu_k$. Let $E$ be the
set of actual poles and $c_\lambda$ their residues. The function
\[
 \B f(\xi)-\sum_{\lambda\in E}\frac{c_\lambda}{\xi-\lambda}
\]
is holomorphic on $\C\setminus\rho K$. Since
$\B(\ee^{\lambda z})(\xi)=(\xi-\lambda)^{-1}$, the carrier
criterion gives \eqref{eq:decomposition}.

If $E$ were empty, we would have $K\subset\rho K$, and comparison
of supports would force $K=\{0\}$. Estimate~\eqref{eq:uniform} would
then make the exponential type of $f$ zero, contradicting
the lower bound at the beginning of the proof. Thus $E\neq\varnothing$. Finally,
\eqref{eq:decomposition} gives
$K\subset\conv(E\cup\rho K\cup\{0\})$. The same support comparison
removes the term $\rho K$ and proves \eqref{eq:hull-control}.
\end{proof}

\section{Polynomial-composition rigidity}\label{sec:rigidity}

We now impose the composition $f=P(h)$. The first consequence is purely
geometric: two characteristic frequencies that form an exposed edge
away from the origin create a forbidden zero count.

\begin{lem}\label{lem:antipodal}
Let $P$ be a polynomial of degree $n\geq2$, let $h$ be entire, and
suppose that the nonconstant function $f=P(h)$ is of finite
exponential type and satisfies \eqref{eq:sublinear} for some $0\leq\rho<1$. In the decomposition
of Proposition~\ref{prop:decomposition}, the set $E$ is contained in
$\{\lambda,-\lambda\}$ for some $\lambda\in\mu_k$.
\end{lem}

\begin{proof}
Suppose otherwise. The finite set $E$ lies on the unit circle. It then
has two consecutive points $\lambda,\mu$, in circular order, whose
intervening empty arc has length strictly less than $\pi$. For two
nonantipodal points take the shorter arc; for at least three points
take a shortest consecutive arc. The chord joining these points is
an exposed edge of $\conv(E\cup\{0\})$. Accordingly there is a
unit complex number $\nu$ such that
\begin{equation}\label{eq:exposed-edge}
 \Rea(\lambda\nu)=\Rea(\mu\nu)=b>0,
 \qquad
 \Rea(\eta\nu)<b\quad
 (\eta\in E\setminus\{\lambda,\mu\}).
\end{equation}
For completeness, rotate the endpoints to $\ee^{\pm\ii\beta}$ with
$0<\beta<\pi/2$. The corresponding support value is $\cos\beta>0$;
all the other points lie on the complementary arc and have smaller
support value. Rotating back gives \eqref{eq:exposed-edge}.

Consider
\[
 B(z)=c_\lambda\ee^{\lambda z}
       +c_\mu\ee^{\mu z}.
\]
Its zeros satisfy
$\ee^{(\lambda-\mu)z}=-c_\mu/c_\lambda$ and are simple.
Since $\Rea((\lambda-\mu)\nu)=0$, one half of this arithmetic
progression can be written
\[
 z_j=z_*+t_j\nu,\qquad t_j\longrightarrow+\infty.
\]
Choose a fixed $r_0>0$ smaller than the distance between consecutive
zeros; see Figure~\ref{fig:exposed-edge-zeros}.

\begin{figure}[htbp]
\centering
\definecolor{edgeblue}{HTML}{2D5E88}
\definecolor{hullteal}{HTML}{178A82}
\definecolor{supportgold}{HTML}{B2792F}
\definecolor{axisgray}{HTML}{98A5AF}

\begin{tikzpicture}[
  x=1.1cm,y=1.1cm,font=\small,
  line cap=round,line join=round,
  axis/.style={
    draw=axisgray,line width=.45pt,
    -{Stealth[length=4pt,width=3.2pt]}
  },
  direction/.style={
    draw=edgeblue,line width=.9pt,
    -{Stealth[length=5pt,width=4pt]}
  },
  zero/.style={draw=edgeblue,line width=1pt},
  every node/.style={inner sep=2pt}
]
\path[use as bounding box] (0,0) rectangle (13.5,5.1);

\begin{scope}[shift={(2.3,2.55)}]
  \coordinate (O) at (0,0);
  \coordinate (L) at (36:1.6);
  \coordinate (M) at (-36:1.6);
  \coordinate (H) at (-1.6,0);
  \coordinate (l) at ($(O)!.58!(L)$);
  \coordinate (m) at ($(O)!.58!(M)$);
  \coordinate (h) at ($(O)!.58!(H)$);

  \fill[edgeblue!7] (H)--(L)--(M)--cycle;
  \fill[hullteal!18] (h)--(l)--(m)--cycle;

  \draw[axisgray!75,densely dotted,line width=.65pt]
    (O) circle[radius=1.6];

  \draw[axis] (-1.95,0)--(2.05,0)
    node[right,text=axisgray] {$\operatorname{Re}\xi$};
  \draw[axis] (0,-1.8)--(0,1.9)
    node[above,text=axisgray] {$\operatorname{Im}\xi$};
  \node[text=axisgray] at (-1.08,1.72) {$|\xi|=1$};

  \draw[supportgold,dash pattern=on 3pt off 2.5pt,
        line width=.8pt]
    ($(L)+(0,.55)$)--($(M)-(0,.55)$);
  \node[anchor=west,text=supportgold]
    at ($(L)+(.12,.65)$) {$\operatorname{Re}(\xi\nu)=b$};

  \draw[edgeblue,line width=.95pt]
    (H)--(L)--(M)--cycle;
  \draw[edgeblue,line width=1.7pt] (L)--(M);
  \draw[hullteal,dash pattern=on 3pt off 2pt,
        line width=.85pt]
    (h)--(l)--(m)--cycle;

  \draw[direction] (.12,.15)--(1.02,.15)
    node[midway,above] {$\nu$};

  \node[text=edgeblue] at (-.45,.77) {$C$};
  \node[text=hullteal] at (.46,-.28) {$\rho C$};
  \node[below left,text=axisgray] at (O) {$0$};

  \foreach \p in {L,M,H}
    \fill[edgeblue] (\p) circle[radius=2pt];

  \node[above right,text=edgeblue] at (L) {$\lambda$};
  \node[below right,text=edgeblue] at (M) {$\mu$};
  \node[above left,text=edgeblue] at (H) {$\eta$};
\end{scope}

\node at (2.6,.55)
  {$C=\operatorname{conv}(E\cup\{0\})$};
\node at (2.6,.05)
  {\textup{(a)} Frequency plane ($\nu=1$)};

\begin{scope}[shift={(6.55,2.55)}]
  \node at (3.2,1.75)
    {$B(z)=c_\lambda e^{\lambda z}+c_\mu e^{\mu z}$};

  \draw[axisgray,densely dashed,line width=.5pt,
        -{Stealth[length=4pt,width=3.2pt]}]
    (-.05,0)--(6.55,0);

  \foreach \x in {1,3.2,5.4}{
    \filldraw[fill=white,draw=edgeblue,line width=1pt]
      (\x,0) circle[radius=.62];
    \fill[hullteal,opacity=.06]
      (\x,0) circle[radius=.62];
  }

  \draw[direction] (.55,1.05)--(1.5,1.05)
    node[midway,above] {$\nu$};

  \draw[supportgold,line width=.8pt]
    (1,0)--(1,.62)
    node[midway,right,text=supportgold] {$r_0$};

  \foreach \x in {1,3.2,5.4}{
    \draw[zero] (\x-.055,-.055)--(\x+.055,.055);
    \draw[zero] (\x-.055,.055)--(\x+.055,-.055);
  }

  \node[text=edgeblue] at (1,-.88) {$z_j$};
  \node[text=edgeblue] at (3.2,-.88) {$z_{j+1}$};
  \node[text=edgeblue] at (5.4,-.88) {$z_{j+2}$};

  \node at (3.2,-1.4) {$|z-z_j|=r_0$};
  \node at (3.2,-2)
    {$z_j=z_*+t_j\nu,\quad t_j\to+\infty$};
\end{scope}

\node at (9.75,.05) {\textup{(b)} Zeros of $B$};

\end{tikzpicture}
\caption{Exposed edge and circles around zeros.}
\label{fig:exposed-edge-zeros}
\end{figure}
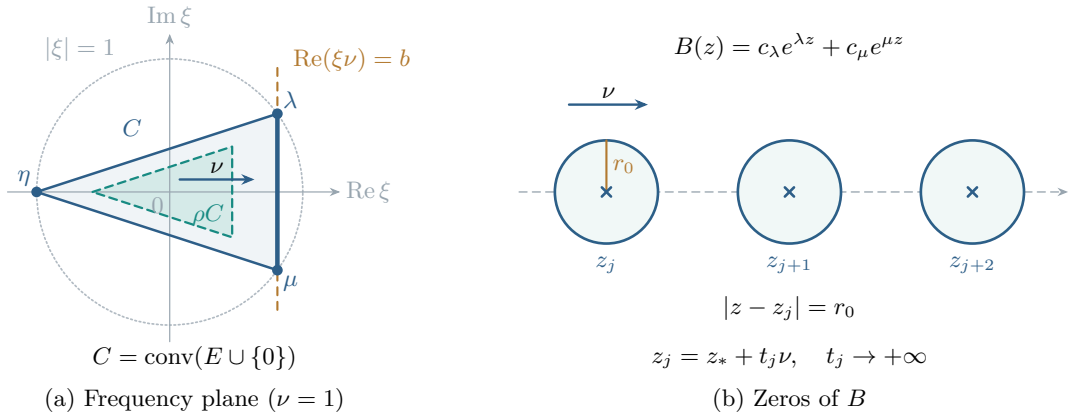

For $|\zeta|=r_0$,
\[
 B(z_j+\zeta)
 =c_\lambda\ee^{\lambda z_j}
       \bigl(\ee^{\lambda\zeta}-\ee^{\mu\zeta}\bigr).
\]
The second factor has no zeros on this circle. Therefore
\begin{equation}\label{eq:boundary-lower}
 |B(z)|\geq c\ee^{bt_j}
 \qquad\bigl(|z-z_j|=r_0\bigr)
\end{equation}
with a constant $c>0$ independent of $j$.

By \eqref{eq:exposed-edge}, all the remaining characteristic terms
in \eqref{eq:decomposition} are $o(\ee^{bt_j})$ on these circles,
uniformly in their angular parameter. To estimate the remainder, use
$K_v\subset\rho K$, \eqref{eq:hull-control}, and
\eqref{eq:uniform}. Since the offsets $z-z_j$ stay bounded, they give
\[
 |v(z)|\leq C_\varepsilon
           \exp\bigl((\rho b+\varepsilon)t_j\bigr)
 \qquad\bigl(|z-z_j|=r_0\bigr).
\]
Choose $0<\varepsilon<(1-\rho)b$. Thus
$f-B=o(\ee^{bt_j})$ uniformly on the circles. From
\eqref{eq:boundary-lower} and Rouch\'e's theorem, $f$ has exactly
one zero, counted with multiplicity, in each disk
$|z-z_j|<r_0$ for all sufficiently large $j$.

On the same boundary circles, $|f|\to\infty$ uniformly. Since
$f=P(h)$ and a polynomial is bounded on each bounded set,
$|h|\to\infty$ uniformly there. Writing $p_n$ for the leading
coefficient of $P$, we consequently have
\[
 |P(h)-p_nh^n|<|p_nh^n|
\]
on each sufficiently distant circle. A second application of
Rouch\'e's theorem shows that the number of zeros of $f$ in the
corresponding disk equals the number of zeros of $h^n$. The latter
is divisible by $n\geq2$, contradicting the count one.
\end{proof}

The conclusion of Lemma~\ref{lem:antipodal} gives a dominant
exponential in a half-plane direction. The remaining step is to show
that its natural period is an exact period of the whole function.

\begin{thm}\label{thm:composition}
Let $h$ be entire, let $P,S\in\C[w]$, and suppose that
$n=\deg P\geq2$ and $\deg S\leq n-1$. The zero polynomial is allowed
for $S$. If $f=P(h)$ is nonconstant and
\begin{equation}\label{eq:composition-recalled}
 f^{(k)}-f=S(h),
\end{equation}
then, for some $\lambda\in\C$ with $\lambda^k=1$,
\begin{equation}\label{eq:laurent}
 f(z)=\sum_{j=-n}^{n}a_j\exp\!\left(\frac{j\lambda z}{n}\right),
 \qquad a_j\in\C.
\end{equation}
If $k$ is odd, $\lambda$ can be chosen so that $a_j=0$ for $j<0$.
In particular, $f\in\W$.
\end{thm}

\begin{proof}
Put $\rho=(n-1)/n$. Comparing polynomial degrees at infinity and
then enlarging the constant on a bounded set gives
\[
 |S(w)|\leq C\bigl(1+|P(w)|^\rho\bigr)
 \qquad(w\in\C).
\]
Together with \eqref{eq:composition-recalled}, this gives
\eqref{eq:sublinear}. The polynomials $P$ and $P+S$ are
algebraically dependent, so $f$ and $f^{(k)}$ satisfy a nonzero
polynomial relation. Proposition~\ref{prop:entire-type} therefore
gives finite exponential type, and
Proposition~\ref{prop:decomposition} applies.
By Lemma~\ref{lem:antipodal}, choose $\lambda\in E$ so that
$E\subset\{\lambda,-\lambda\}$. Replacing $f(z),h(z)$ by
$f(z/\lambda),h(z/\lambda)$ preserves
\eqref{eq:composition-recalled}, since $\lambda^k=1$. We may thus suppose
that $\lambda=1$, and write
\begin{equation}\label{eq:two-terms}
 f(z)=c_+\ee^z+c_-\ee^{-z}+v(z),
 \qquad c_+\neq0,\qquad K_v\subset[-\rho,\rho].
\end{equation}
Also $K_f\subset[-1,1]$ by \eqref{eq:hull-control}. If $k$ is odd,
$-1\notin\mu_k$, so $c_-=0$ and the stronger inclusion
$K_f\subset[0,1]$ holds.

\smallskip
\noindent\emph{A coherent root asymptotic.}
On every fixed horizontal strip, \eqref{eq:uniform} and
\eqref{eq:two-terms} give
$f(z)=c_+\ee^z(1+o(1))$ as $\Rea z\to+\infty$, uniformly in the
imaginary part. In particular $|h(z)|\to\infty$ uniformly there.
If $p_n$ is the leading coefficient of $P$, then
\[
 \bigl(h(z)\ee^{-z/n}\bigr)^n\longrightarrow c_+/p_n
\]
uniformly on the strip. The $n$ roots of $c_+/p_n$ are distinct.
For sufficiently large real part, the continuous function
$h(z)\ee^{-z/n}$ lies in the union of disjoint small disks about
these roots. The right-hand part of the strip is connected, so its
image lies in only one of the disks. Shrinking the disks proves that
there is a single $\gamma$ for that strip such that
\begin{equation}\label{eq:root-asymptotic}
 h(z)=\gamma\ee^{z/n}(1+o(1)),
 \qquad \gamma^n=c_+/p_n.
\end{equation}
This establishes the consistency of the root; no branch is selected
independently on separate lines.

\smallskip
\noindent\emph{A translation with a small diagram.}
Let
\[
 T=2\pi\ii n,\qquad u(z)=f(z+T)-f(z).
\]
The two explicit exponentials in \eqref{eq:two-terms} disappear under
this subtraction. Translation invariance of convex carriers gives
\begin{equation}\label{eq:small-diagram}
 K_u\subset[-\rho,\rho]\subset\D.
\end{equation}
Choose a fixed horizontal strip containing both lines $\Im z=0$
and $\Im z=2\pi n$ in its interior. Using
\eqref{eq:root-asymptotic} in this one strip yields
\begin{equation}\label{eq:root-ratio}
 \frac{h(x+T)}{h(x)}\longrightarrow1
 \qquad(x\longrightarrow+\infty),
\end{equation}
since $\ee^{T/n}=1$.

For any polynomial $A$ define its polynomial divided difference by
\[
 \Delta A(s,t)=\frac{A(s)-A(t)}{s-t},
 \qquad \Delta A(t,t)=A'(t).
\]
Subtracting the two translated copies of
\eqref{eq:composition-recalled} gives, whenever the denominator below is
nonzero,
\begin{equation}\label{eq:translation-equation}
 Lu(z)=a_T(z)u(z),\qquad
 a_T(z)=\frac{\Delta S(h(z+T),h(z))}
                 {\Delta P(h(z+T),h(z))}.
\end{equation}
This identity is valid even where $h(z+T)=h(z)$, because the divided
differences are polynomials: no division by $h(z+T)-h(z)$ is needed.

By \eqref{eq:root-ratio} and \eqref{eq:root-asymptotic},
\[
 \Delta P(h(x+T),h(x))
       =np_nh(x)^{n-1}(1+o(1)).
\]
Thus this denominator is nonzero for all sufficiently large real
$x$. Since $\deg S\leq n-1$, its divided difference has total degree
at most $n-2$ unless it vanishes identically. It follows that
\begin{equation}\label{eq:decaying-coefficient}
 |a_T(x)|\leq C|h(x)|^{-1}
           \leq C'\ee^{-x/n}
 \qquad(x\geq x_0).
\end{equation}

Assume that $u\neq0$. Lemma~\ref{lem:operator} and
\eqref{eq:small-diagram} show that $Lu\neq0$ and $K_{Lu}=K_u$.
By the exact radial formula~\eqref{eq:indicator}, the finite real
number $\sigma=H_{K_u}(1)$ therefore satisfies
\[
 \limsup_{x\to\infty}\frac{\log|u(x)|}{x}
 =\sigma
 =\limsup_{x\to\infty}\frac{\log|Lu(x)|}{x}.
\]
On the other hand, \eqref{eq:translation-equation} and
\eqref{eq:decaying-coefficient} imply that the rightmost upper limit
is at most $\sigma-1/n$. This contradiction proves $u=0$. Hence
$f$ has period $2\pi\ii n$.

\smallskip
\noindent\emph{From a period to a finite Laurent expansion.}
Periodicity defines a holomorphic function $R$ on $\C^*$ by
$f(z)=R(\ee^{z/n})$. Choose a logarithm of $w\in\C^*$ with imaginary
part in $[0,2\pi)$. Finite exponential type then gives, for suitable
constants $A,B>0$,
\[
 |R(w)|\leq A\max\{|w|^B,|w|^{-B}\}.
\]
Consequently $0$ and $\infty$ are at most poles of $R$, so $R$ is a
Laurent polynomial. Its nonzero coefficient at $w^j$ corresponds
to a pole of $\B f$ at $j/n$. Since $K_f\subset[-1,1]$, only
$-n\leq j\leq n$ can occur. When $k$ is odd, the stronger inclusion
$K_f\subset[0,1]$ excludes the negative indices. Undoing the change
$z\mapsto z/\lambda$ proves \eqref{eq:laurent} and the last
assertion of the theorem.
\end{proof}

Two features of the argument merit emphasis. The estimates near the
zeros in Lemma~\ref{lem:antipodal} are uniform on circles of a fixed
radius; a merely pointwise radial estimate would not suffice there.
In the translation argument, only the identity between the ordinary
upper indicator and the support function is used. No assertion of
completely regular growth, and no pointwise lower growth bound on
$u$, is required.

\section{Completion of the classification}\label{sec:classification}

We now return to arbitrary meromorphic solutions. The analytic theorem
just proved closes precisely the polynomial-parametrization case of
Section~\ref{sec:reduction}.

\begin{equation}\label{eq:linear}
 y^{(k)}+ay+b=0,\qquad a,b\in\C;
\end{equation}

\begin{thm}\label{thm:main}
Let $k\geq1$, and let $y$ be a meromorphic function on $\C$ satisfying
\eqref{eq:bb}. Then either $y\in\W$ or $y$ satisfies
\eqref{eq:linear} for some constants $a,b\in\C$. Conversely, every
function in $\W$ and every meromorphic solution of \eqref{eq:linear}
is algebraically dependent over $\C$ with its $k$th derivative.
\end{thm}

\begin{proof}
A constant or a polynomial belongs to $\W$. If $y$ has a pole, the
pole theorem \cite[Theorem~1]{ELN09} of Eremenko, Liao, and Ng applies, giving $y\in\W$. It remains to consider a
transcendental entire function $y\notin\W$.

By Corollary~\ref{cor:entire-poles}, $\Gamma$ has genus zero and
$|V|\in\{1,2\}$. If $|V|=2$, Corollary~\ref{cor:two-poles}
gives $y\in\W$, contrary to our assumption. Thus $|V|=1$,
and Section~\ref{sec:reduction} gives polynomial parametrizations
$y=P(h)$ and $y^{(k)}=Q(h)$ of the same degree.

If this degree were at least two, the rescaling described there
and Theorem~\ref{thm:composition} would give $y\in\W$.
The common degree is therefore one, and \eqref{eq:linear} follows.

For the converse, a solution of \eqref{eq:linear} already satisfies
a polynomial relation. If $y\in\C(z)$, then $y^{(k)}\in\C(z)$;
if $y=R(\ee^{az})$, then repeated application of
$a w\,\dd/\dd w$ shows that $y^{(k)}$ is rational in the same
exponential. Two rational functions of one variable are algebraically
dependent.

The same conclusion for elliptic functions can be seen without an
explicit Weierstrass equation. On the torus of a common period lattice,
let $D_0$ be an effective divisor bounding the poles of $y$ and
$y^{(k)}$. The functions $y^i(y^{(k)})^j$ with $i+j\leq N$ have
poles bounded by $ND_0$. The space of functions with that pole bound
has dimension at most $N\deg D_0+1$: their principal parts determine
them up to an additive constant, since a holomorphic function on a
compact surface is constant. For large $N$, the number
$(N+1)(N+2)/2$ of these monomials exceeds that bound. A nontrivial
linear relation between them is the required polynomial relation.
\end{proof}

Theorem~\ref{thm:main} has the following consequence for
irreducible equations.

\begin{cor}\label{cor:nonlinear}
If $F\in\C[X,Y]$ is irreducible of total degree at least two, then
every meromorphic solution of \eqref{eq:bb} belongs to $\W$.
\end{cor}

\begin{proof}
Suppose that a solution $y$ does not belong to $\W$. It is nonconstant,
and Theorem~\ref{thm:main} gives $y^{(k)}=-ay-b$. Hence
\[
 F(y,-ay-b)=0.
\]
The one-variable polynomial $F(X,-aX-b)$ vanishes on the infinite
range of $y$ and is therefore identically zero. Polynomial division
in $Y$ shows that $Y+aX+b$ divides $F(X,Y)$. Irreducibility forces
$F$ to have total degree one, contrary to the hypothesis.
\end{proof}
\begin{proof}[Proof of Theorem~\ref{thm:pdf-main}]
Let $y$ be a transcendental entire solution of \eqref{eq:bb},
where $F$ is irreducible of total degree at least two.
By Corollary~\ref{cor:nonlinear}, we have $y\in\W$.
An entire rational function is a polynomial, and an entire
elliptic function is constant. Both possibilities are excluded
by the transcendence of $y$. Hence
\[
 y(z)=R(e^{cz})
\]
for some rational function $R$ and some $c\in\C^*$.
\end{proof}

We now solve \eqref{eq:linear} and determine which solutions
belong to $\W$.

\begin{prop}\label{prop:explicit}
Put $\omega=\exp(2\pi\ii/k)$. The meromorphic functions on $\C$
that are algebraically dependent over $\C$ with their $k$th
derivatives form exactly the set
\begin{equation}\label{eq:full-family}
 \W\ \cup\
 \left\{
 C+\sum_{j=0}^{k-1}C_j\exp(\lambda\omega^j z):
 \lambda\in\C\setminus\{0\},\ C,C_j\in\C
 \right\}.
\end{equation}
A nonconstant member of the exponential family in
\eqref{eq:full-family} belongs to $\W$ if and only if its nonzero
frequencies with nonzero coefficients are contained in an antipodal
pair $\{\eta,-\eta\}$.
\end{prop}

\begin{proof}
A meromorphic solution of \eqref{eq:linear} cannot have a pole:
the pole order of $y^{(k)}$ would be strictly larger than that of
$ay+b$. If $a=0$, integration shows that $y$ is a polynomial of
degree at most $k$, already in $\W$. If $a\neq0$, choose
$\lambda^k=-a$ and put $C=-b/a$. The characteristic roots
$\lambda\omega^j$ are distinct. Their exponentials give all
solutions of the homogeneous equation: their derivatives of orders
$0,\ldots,k-1$ at zero form a nonsingular Vandermonde matrix, and
the differential equation determines all higher Taylor coefficients
from these initial data. Thus the solutions are precisely the
functions displayed in \eqref{eq:full-family}. Conversely, each
such function satisfies
$y^{(k)}-\lambda^k y+\lambda^k C=0$.

For the overlap statement, discard the terms with zero coefficient
and write the remaining distinct frequencies as
$\eta_1,\ldots,\eta_s$, with coefficients $d_1,\ldots,d_s\neq0$.
Distinct exponentials are linearly independent: differentiating
orders $0,\ldots,s-1$ at zero again gives a Vandermonde matrix.
A nonzero period $T$ would therefore imply
\begin{equation}\label{eq:period-frequencies}
 \ee^{\eta_\ell T}=1\qquad(1\leq\ell\leq s).
\end{equation}
For any two frequencies, \eqref{eq:period-frequencies} makes their
ratio real and rational. All the frequencies have the same modulus,
so their ratio must be $1$ or $-1$. They consequently lie in one
antipodal pair.

A nonconstant exponential sum here cannot be rational: if it were,
being entire it would be a polynomial, whereas a sufficiently high
derivative and linear independence of the exponentials rule this
out. Nor can it be elliptic, since a nonconstant elliptic function
has a pole. Thus membership in $\W$ implies representation as a
rational function of an exponential and hence a nonzero period;
the preceding argument applies. Conversely, frequencies in one
antipodal pair give a Laurent polynomial in $\ee^{\eta z}$, which
belongs to $\W$.
\end{proof}

The low-order classifications and the necessity of the exceptional
family are now immediate, with no additional assumptions on the
polynomial relation.

\begin{cor}\label{cor:low-order}
For $k=1$ or $k=2$, every meromorphic solution of \eqref{eq:bb}
belongs to $\W$. For each $k\geq3$, there are entire solutions of
\eqref{eq:bb} outside $\W$.
\end{cor}

\begin{proof}
For $k=1,2$, the characteristic frequencies in
Proposition~\ref{prop:explicit} are contained in one antipodal pair.
For $k\geq3$, take $\omega=\exp(2\pi\ii/k)$ and
\[
 y(z)=\ee^z+\ee^{\omega z}.
\]
Then $y^{(k)}=y$, whereas the two frequencies are not antipodal.
Proposition~\ref{prop:explicit} therefore gives $y\notin\W$.
\end{proof}

The first assertion recovers the classical first- and second-order
theorems. The second shows that the linear clause in
Theorem~\ref{thm:main} is indispensable. More structurally, the proof
identifies the only obstruction to the classical classification:
a solution outside $\W$ forces the irreducible algebraic relation
between the function and its derivative to be a line.


\end{document}